\documentclass[12pt]{amsart}
\usepackage{amsmath, amsthm, amscd, amsfonts}
\usepackage{graphicx}

\newfont{\aj}{eufm10 at12pt}
\newfont{\ajk}{eufm10 at10pt}
\theoremstyle{plain}
\newtheorem{theorem}{Theorem}[section]
\newtheorem{lemma}[theorem]{Lemma}

\newtheorem{corollary}[theorem]{Corollary}
\theoremstyle{definition}
\newtheorem{definition}[theorem]{Definition}
\newtheorem{example}[theorem]{Example}

\newtheorem{remark}[theorem]{Remark}
\numberwithin{equation}{section}
\begin{document}
\title[Chain Mixing and Chain Recurrent Iterated Function Systems]{ Chain
Mixing and Chain Recurrent Iterated Function Systems}
\author[Mehdi Fatehi Nia]{ {{\bf {Mehdi Fatehi Nia}}
\\{\tiny{Department of Mathematics, Yazd University, 89195-741 Yazd, Iran}
\\e-mail: fatehiniam@yazd.ac.ir }}}
\maketitle
\vspace{5mm}
\begin{abstract}
This paper considers the egodicity properties in iterated function systems. First, we will introduce chain mixing and chain transitive iterated function systems  then some results and examples are presented to compare with these notions in discrete dynamical systems. As our main result, using adding machine maps and topological conjugacy we show that chain mixing, chain transitive and chain recurrence properties in iterated function systems are equivalent.
\end{abstract}
\emph{keywords}: { Adding machine maps, chain mixing, chain recurrent, chain transitive, iterated function systems, shadowing property}\\
\emph{subjclass}[2010]{ 37C50,37C15}
\section{Introduction}
Let $(X,d)$ be a compact Hausdorff metric space. Let us recall that an \emph{ Iterated Function System(IFS)} $\mathcal{F}=\{X; f_{\lambda}|\lambda\in\Lambda\}$ is any family of continuous mappings $f_{\lambda}:X\rightarrow X,~\lambda\in \Lambda$, where $\Lambda$ is a finite nonempty set (see\cite{[GG]}).\\ Let $T=\mathbb{Z}$ or $T=\mathbb{Z}_{+}= \{n\in \mathbb{Z}:n\geq 1\}$ and $\Lambda^{T}$ denote the set
of all infinite sequences $\{\lambda_{i}\}_{i\in T}$ of symbols belonging to $\Lambda$. A typical element of $\Lambda^{\mathbb{Z}_{+}}$
 can be denoted as $\sigma= \{\lambda_{1},\lambda_{2},...\}$ and we use the shorted notation $$\mathcal{F}_{\sigma_{n}}=f_{\lambda_{n}} o f_{\lambda_{n-1}}o ...o f_{\lambda_{1}}.$$ A sequence $\{x_{n}\}_{n\in T}$ in $X$ is called an orbit of the \textbf{IFS} $\mathcal{F}$ if there exists $\sigma\in \Lambda^{T}$ such that $x_{n+1}=f_{\lambda_{n}}(x_{n})$, for each $\lambda_{n}\in \sigma$.\\Many  notions in
dynamics like attractors, minimality, transitivity, and shadowing can be extended to
IFS (see \cite{[BV],[BVA],[MF],[GG1]}).\\
Topological mixing and topological transitivity are two important notions in discrete dynamical system. In \cite{[AK]}, Akin E. gave the definition of topological ergodicity, and prove that topological mixing implied topological ergodicity and topological ergodicity implied topological transitivity.
 Many authors  study different properties of topological ergodicity and their applications in other areas\cite{[M],[RW],[KS]}. For example  chain recurrence and  chain transitivity have applications in
the study of  epidemiology, game theory, economics, and mathematical biology (see \cite{[M]} for more details). In \cite{[RW]}, Richeson and Wiseman prove that for every continuous map on a connected space, chain mixing,
chain transitivity and chain recurrent properties are equivalent. Up to this point, they prove that the growth of the chain mixing
times give some information about topological entropy.
The main goal of this paper is to extend the following result to iterated function systems:
\begin{theorem}\label{taa}\cite{[RW]}
Let $X$ be connected and $f:X\rightarrow X$ be continuous. Then the following are equivalent:\\
$(1)~f$ is chain recurrent. \\
$(2)~f$ is chain transitive.\\
$(3)~f$ is totally chain transitive.\\
$(4)~f$ is chain mixing.
\end{theorem}
For this purpose, firstly, we define the notions chain mixing and chain transitive iterated function systems.
Then by some lemmas and examples we present some results about these notions on iterated function systems.
Lemma \ref{led} is one of the main results of this paper, this lemma present some conditions for an \textbf{IFS}
to be topologically conjugated with an adding machine map.  In Theorem \ref{main} we prove that if $\mathcal{F}$ is chain transitive then for some $k>1$,$\mathcal{F}^{k}$ restricted to some invariant subsets  is chain mixing; or
 $\mathcal{F}$ factors to an adding machine map.
As a corollary we have Theorem \ref{ca} which is a generalization of  Theorem \ref{taa} for iterated function systems. Finally
we consider some results that are related to this theorem.
\section{preliminaries and definitions}
 At first we establish some notations.\\ Given $\delta>0$, a sequence $\{x_{n}\}_{n\in T}$ in $X$ is called a $\delta-$chain of $\mathcal{F}$ if there exists $\sigma\in\Lambda^{T}$ such that for every $\lambda_{n}\in \sigma$, we have $d(x_{n+1},f_{\lambda_{n}}(x_{n}))<\delta$.\\
A point $x\in X$ is chain recurrent if for every $\epsilon>0$, there is an $\epsilon-$chain from $x$ to itself \cite{[BWL]}.
\\
An IFS $\mathcal{F}$ is chain transitive if for every $x,y\in X$ and every $\epsilon>0$ there is an $\epsilon-$chain from $x$ to $y$.\\
 \\One says that the \textbf{IFS} $ \mathcal{F}$ has the \emph{shadowing property }(on $T$) if, given $\epsilon>0$, there exists $\delta>0$ such that for any $\delta-$pseudo orbit $\{x_{n}\}_{n\in T}$ there exists an orbit $\{y_{n}\}_{n\in T}$, satisfying the inequality $d(x_{n},y_{n})\leq \epsilon$ for all $n\in T$. In this case one says that the $\{y_{n}\}_{n\in T}$ or the point $y_{0}$, $\epsilon-$ shadows the $\delta-$pseudo orbit $\{x_{n}\}_{n\in T}$.\\
Please note that if $\Lambda$ is a set with one member then the \textbf{IFS} $ \mathcal{F}$ is an ordinary discrete dynamical system \cite{[GG]}.
In this case the shadowing property for $\mathcal{F}$ is ordinary shadowing property for a discrete dynamical system.

\begin{definition}
Suppose $(X,d)$ and $(Y,d^{'})$ are compact metric spaces and $\Lambda$ is a finite set. Let $ \mathcal{F}=\{X; f_{\lambda}|\lambda\in\Lambda\}$ and $ \mathcal{G}=\{Y; g_{\lambda}|\lambda\in\Lambda\}$ are two $IFS$ which $f_{\lambda}:X\rightarrow X$ and  $g_{\lambda}:Y\rightarrow Y$ are continuous maps for all $\lambda\in\Lambda$. We say that $ \mathcal{F}$ is  topologically conjugate to $ \mathcal{G}$  if there is a homeomorphism  $h:X\rightarrow Y$  such that $g_{\lambda}=h o f_{\lambda} o h^{-1}$, for all $\lambda\in\Lambda$.\end{definition}

\begin{definition}
The IFS $\mathcal{F}$ is topologically transitive if for every nonempty sets $U,V\subset X$, there is $\sigma\in \Lambda^{\mathbb{Z}_{+}}$ such that $\mathcal{F}_{\sigma_{n}}(U)\cap V\neq\emptyset$, for some $n\geq 0$
\end{definition}
\begin{definition}
The IFS $\mathcal{F}$ is topologically mixing if for every nonempty sets $U,V\subset X$, there is $N>0$ such that for all $n>N$ we can find $\sigma\in \Lambda^{\mathbb{Z}_{+}}$ which $\mathcal{F}_{\sigma_{n}}(U)\cap V\neq\emptyset$.
\end{definition}

Let $n>0$ be an integer. Set $\mathcal{F}^{n}= \{ g_{\mu}|\mu\in\Pi\}=\{f_{\lambda_{k}}o ...of_{\lambda_{1}}|\lambda_{1},...,\lambda_{k}\in\Lambda\}$. \\
We say that $\mathcal{F}$ is totally chain transitive if $\mathcal{F}^{n}$ is chain transitive, for all $n\geq 1$.
\begin{definition}
We say that $\mathcal{F}$  is $\epsilon-$chain mixing, if there is an $N>0$ such that for every  $x,y\in X$ and $n>N$ there is an $\epsilon-$chain from $x$ to $y$ of length exactly $n$.
\end{definition}
The IFS $\mathcal{F}$ is chain mixing if it is $\epsilon-$chain mixing, for every $\epsilon>0$.
\section{ergodicity in iterated function  systems}
In this section we will recall and extend some well-known general results in discrete dynamical systems for \textbf{IFS}s. Also, we give two non trivial examples of \textbf{IFS}s that are chain mixing. \begin{lemma}\label{le}
Suppose that the IFS $\mathcal{F}$ has the shadowing property then $\mathcal{F}$ is topologically mixing if and only if it is chain mixing.
\end{lemma}
\begin{proof}
Let $\mathcal{F}$ be topologically mixing. Given $\epsilon>0$ and $a,~b\in X$. So, there exists $N>0$ such that for every $n>N$ we can find $\sigma\in \Lambda^{\mathbb{Z}_{+}}$, which $\mathcal{F}_{\sigma_{n}}(B_{\epsilon}(a))\cap B_{\epsilon}(b)\neq\emptyset.$ Then, there is a chain of  length $n$ from a point in $B_{\epsilon}(a)$ to a point in $B_{\epsilon}(b)$. Thus, there is an $\epsilon-$chain from $a$ to $b$ of  length $n+2$.
\\
Conversely, suppose that $\mathcal{F}$ is chain mixing and has the shadowing property. Assume that $U$ and $V$ are two arbitrary non-empty sets. Take $a\in U$ and $b\in V$ and $\epsilon>0$ such that $B_{\epsilon}(a)\subset U$ and $B_{\epsilon}(b)\subset V$. Consider $\delta>0$ as $\epsilon$ modulus of shadowing. Since $\mathcal{F}$ is chain mixing, there exists $N>0$ such that for every $n>N$, we can find $\sigma\in \Lambda^{\mathbb{Z}_{+}}$ which $x, \mathcal{F}_{\sigma_{1}}(x),...,\mathcal{F}_{\sigma_{n}}(x),y$ is a $\delta-$chain from $x$ to $y$. This implies that, for every $n>N$, we can find $\sigma\in \Lambda^{\mathbb{Z}_{+}}$ which $\mathcal{F}_{\sigma_{n}}(U)\cap V\neq\emptyset$.\\
\end{proof}
This is well known that in discrete dynamical systems topologically mixing implies topologically transitive but not the converse \cite{[BT],[RW]}. For example "an irrational rotation of the circle is topologically transitive (the orbit of a
small open interval will eventually intersect any other small open interval),
but not topologically mixing (a rotating small interval will typically leave
the another interval for a while before returning)"\cite{[BT]}. In the next example we give an \textbf{IFS} consists two irrational rotation that is topological mixing.
\begin{example}\ref{eza}
Consider a circle $S^{1}$ with coordinate $x \in [0; 1)$ and we denote by $d$
the metric on $S^{1}$ induced by the usual distance on the real line. Let $\pi(x) :\mathbb{ R} \rightarrow S^{1}$ be the covering projection defined by the relations
$$\pi(x) \in [0; 1) and~ \pi(x) = x(x ~mod~ 1)$$
with respect to the considered coordinates on $S^{1}$.
Let $\pi:[0,1]\rightarrow S^{1}$ be a map defined by $\pi(t)=(\cos(2\pi t),\sin(2\pi t))$
 Let $F_{1}:[0,1]\rightarrow[0,1]$ be a homeomorphism defined by
 $F_{1}(t) =t+\alpha$
and
 $F_{2}:[0,1]\rightarrow[0,1]$ be a homeomorphism defined by
 $F_{2}(t) =t+\beta$, where $\alpha<\beta$ are two distinct positive  numbers that $\beta-\alpha$ is irrational.\\
Assume that $f_{i}$ is a homeomorphism on $S^{1}$ defined by\\ $f_{i}(\cos(2\pi t),\sin(2\pi t))=(\cos(2\pi F_{i}(t)),\sin(2\pi F_{i}(t)))$, for $i\in\{0,1\}$. \\Consider  the \textbf{IFS} $\mathcal{F}=\{S^{1}, f_{\lambda}|\lambda\in\{1,2\}\}$.  We show that $\mathcal{F}$ is chain mixing.\\Fix $\epsilon>0$. For each sequence $\sigma\in \Lambda^{\mathbb{Z}_{+}}$ and $n>1$,  $\mathcal{F}_{\sigma_{n}}(x)=x+k\alpha+(n-k)\beta$, when $k$ is the cardinality of the sets $\{\lambda_{i}\mid \lambda_{i}=1, 1\leq i\leq n\}$. Since $\beta-\alpha$ is irrational, similar argument to Example $2$ in \cite{[PD]} shows that the set $\{k\alpha+(n-k)\beta~mod~1\mid n>1,~k\leq n\}$ is dense in $[0,1]$ .
 So, for every $x,y$ in $S^{1}$  we can find $N_{0}>1$ such that there exists an $\epsilon-$chain from $x$ to $y$ of  length $n$, for all $n>N_{0}$.
\end{example}
Let $\Lambda$ be a finite set, $ \mathcal{F}=\{X; f_{\lambda}|\lambda\in\Lambda\}$ is an $IFS$ and let $k\geq 2$ be an integer. Set $$\mathcal{F}^{k}= \{X; g_{\mu}|\mu\in\Pi\}=\{X; f_{\lambda_{k-1}}o ...of_{\lambda_{0}}|\lambda_{0},...,\lambda_{k-1}\in\Lambda\}.$$
 The proof of the following lemma is straightforward and is therefore omitted.
\begin{lemma}
Suppose that the \textbf{IFS} $ \mathcal{F}$ is chain mixing then  $ \mathcal{F}^{n}$ is chain transitive, for all positive integers $n$.
\end{lemma}
Given two complete metric spaces $(X, d_X)$ and $(Y, d_Y)$, consider the product set $X \times Y$ endowed with the metric \[D((x_1,y_1),(x_2,y_2))=\max\{d_X (x_1, x_2), d_Y (y_1, y_2)\}\].
Let $\mathcal{F} = \{X;f_{\lambda}|\lambda\in \Lambda\}$ and $\mathcal{G}=\{Y; g_{\gamma}|\gamma\in\Gamma\}$ be two parameterized \textbf{IFS}. The \textbf{IFS} $\mathcal{H} = \{X\times Y ; \Phi_{\lambda,\gamma} |\lambda\in\Lambda, \gamma\in\Gamma\}$, defined by $\Phi_{\lambda,\gamma}(x,y) := (f_{\lambda}(x),g_{\gamma}(y))$ is called the product of the two \textbf{IFS} $\mathcal{F}$ and $\mathcal{G}$.\\
\begin{lemma}\label{lex}
 Suppose that $ \mathcal{F}^{n}$ is chain transitive, for all positive integers $n$. Then the \textbf{IFS}   $ \mathcal{F}\times  \mathcal{F}$ is chain transitive.
\end{lemma}
\begin{proof}
Fix $\epsilon>0$. Let $(a_{1},b_{1})$ and $(a_{2},b_{2})$ be two arbitrary points in $X\times X$. This is sufficient to prove that, there exist two $\epsilon-$chains of the same length from $a_{1}$ to $b_{1}$ and from $a_{2}$ to $b_{2}$.\\
Let $P_{1}=\{a_{1}=x_{0},x_{1},...,x_{m}=b_{1}\}$ be an $\epsilon-$chain from $a_{1}$ to $b_{1}$ and  $P_{2}=\{b_{1}=y_{0},y_{1},...,y_{k}=b_{1}\}$ be an $\epsilon-$chain from $b_{1}$ to itself.\\
Consider a sequence $\{\lambda_{1},...,\lambda_{m-1}\}$ in $\Lambda$. Since $ \mathcal{F}^{k+1}$ is chain transitive, there is an $\epsilon-$chain of $ \mathcal{F}^{k+1}$, $$P^{k+1}_{3}=\{f_{\lambda_{m-1}}o~....o~f_{\lambda_{1}}(a_{2})=z_{0},z_{1},...,z_{l}=b_{2}\}.$$ For each $0\leq i<l$ there exists a sequence $\{\lambda^{i}_{1},...,\lambda^{i}_{k}\}$ in $\Lambda$ such that $d(f_{\lambda^{i}_{k}}~o...of_{\lambda^{i}_{1}}(z_{i}),z_{i+1})<\epsilon$.
Then,  \\
$\{a_{2},f_{\lambda_{1}}(a_{2}),...,f_{\lambda_{m-1}}o~....o~f_{\lambda_{1}}(a_{2}),\\
z_{1},f_{\lambda^{1}_{1}}(z_{1}),...,f_{\lambda^{1}_{k-1}}~o~f_{\lambda^{1}_{k-2}}~o...of_{\lambda^{1}_{1}}(z_{1}),\\
z_{2},f_{\lambda^{2}_{1}}(z_{2}),...,f_{\lambda^{2}_{k-1}}~o~f_{\lambda^{2}_{k-2}}~o...of_{\lambda^{2}_{1}}(z_{2}),\\
\vdots\\
z_{l-1},f_{\lambda^{l-1}_{1}}(z_{l-1}),...,f_{\lambda^{l-1}_{k-1}}~o~f_{\lambda^{l-1}_{k-2}}~o...of_{\lambda^{l-1}_{1}}(z_{l-1}),z_{l}=b_{2}\}$\\
is an $\epsilon-$chain for $ \mathcal{F}$ of  length $m+kl+1$ from $a_{2}$ to $b_{2}$.\\
On the other hand, the sequence $\{P_{1},P_{2},...,P_{2}\}$, which the sequence $P_{2}$ repeated $l$ times is an $\epsilon-$chain of the length  $m+kl+1$ from $a_{1}$ to $b_{1}$. So, the proof is complete.
\end{proof}
Consider the \textbf{IFS} $\mathcal{F} = \{X;f_{\lambda}|\lambda\in \Lambda\}$, by definitions, if $f_{\lambda}$ is chain mixing, for some $\lambda\in \Lambda$, then $\mathcal{F} $ is chain mixing. In the next example we introduce an \textbf{IFS} $\mathcal{F} = \{X;f_{\lambda}|\lambda\in \Lambda\}$ such that $f_{\lambda}$ isn't chain mixing, for any $\lambda\in \Lambda$, but $\mathcal{F}$ is chain mixing.
\begin{example}
Consider the tent map
\[T(x) = \left\lbrace
  \begin{array}{c l}
 2x& \text{if ~$~0 \leq x<\frac{1}{2}$},\\
    -2x+2 & \text{if ~$~\frac{1}{2}\leq x\leq 1$}
 \end{array}
  \right. \]
  By Lemma 3.4 and the last example in \cite{[FH]}, $T(x)$ is chain mixing.\\
  Let $f_{1}(x):[0,1]\rightarrow [0,1]$ be a map defined by

\[ f_{1}(x) = \left\lbrace
  \begin{array}{c l}
 T(x)& \text{if ~$~0 \leq x<\frac{1}{2}$},\\
    1& \text{if ~$~\frac{1}{2}\leq x\leq 1$}
 \end{array}
  \right. \]
  and
  $f_{2}(x):[0,1]\rightarrow [0,1]$ be a map defined by

\[ f_{2}(x) = \left\lbrace
  \begin{array}{c l}
 1& \text{if ~$~0 \leq x<\frac{1}{2}$},\\
  T(x)& \text{if ~$~\frac{1}{2}\leq x\leq 1$}
 \end{array}
  \right. \]
  Consider $0<\epsilon< \frac{1}{2}$, since $f_{1}(\frac{3}{4})=1$ then there isn't any $\epsilon-$chain from $\frac{3}{4}$ to $1$. This implies that $f_{1}$, as an ordinary dynamical systems,  isn't chain mixing. Similar argument shows that  $f_{2}$ isn't chain mixing. \\But, we show that $\mathcal{F}=\{f_{1}, f_{2}; [0,1]\}$ is a chain mixing \textbf{IFS}.\\
  Let $\epsilon>0$ be arbitrary. Since $T(x)$ is chain mixing, there is $N>0$ such that for any $x,y\in X$ and $n>N$ there exists an $\epsilon-$chain for tent map from $x$ to $y$ of the length exactly $n$. \\
  Suppose that $x=x_{0},x_{1},...,x_{n-1}=y$ is the corresponding $\epsilon-$chain for tent map. Consider the sequence $\sigma= \{\lambda_{0},\lambda_{1},..., \lambda_{n-1}\}$, where $\lambda_{0}=1$ if $x_{i}\in [0,\frac{1}{2}]$ and $\lambda_{0}=1$ if $x_{i}\in [\frac{1}{2},1]$. So, for all $0\leq i\leq n-1$, we have $f_{\lambda_{i}}(x_{i}) =T(x_{i})$ and hence $d(f_{\lambda_{i}}(x_{i}),x_{i+1})=d(T(x_{i}),x_{i+1})<\epsilon$. This implies that $x=x_{0},x_{1},...,x_{n-1}=y$ is an $\epsilon-$chain from $x$ to $y$ of  length exactly $n$ for $\mathcal{F}$.
\end{example}
\begin{lemma}\label{lea}
Let $\mathcal{F}=\{X; f_{\lambda}|\lambda\in\Lambda\}$ be an \textbf{IFS} and $X$ be a connected space. Chain recurrence implies chain transitivity.
\end{lemma}
\begin{proof}
Assume that $\mathcal{F}$ has the chain recurrence property and pick $\epsilon>0$. We say that $x$ and $y$ are $\epsilon-$chain equivalent, if there are $\epsilon-$chains from $x$ to $y$ and from $y$ to $x$. Since $\mathcal{F}$ is chain recurrent, this is an equivalence relation. By connectedness   of $X$ this is sufficient to show that every equivalence classes is an open set.\\
Let $x,y$ are $\epsilon-$chains equivalent. We show that there is $\delta>0$ such that $x$ is $\epsilon-$chains equivalent to all points of $B_{\delta}(y)$. Choose $\delta\leq \frac{\epsilon}{2}$ such that $d(a,b)<\delta $ implies that $d(f_{\lambda}(a),f_{\lambda}(b))<\frac{\epsilon}{2}$, for all $a,b \in X$ and $\lambda\in\Lambda$. Consider the arbitrary point $z\in B_{\delta}(y)$, we show that $x$ is $\epsilon-$chains equivalent to $z$. Let $(y_{0}=x,y_{1},..,y_{n}=y) $ be an $\epsilon-$chain from $x$ to $y$ and $(y_{n}=y, y_{n+1},...,y_{n+m}=y)$ be an $\epsilon-$chain from $y$ to itself. Then, $(y_{0}=x,y_{1},..,y_{n}=y, y_{n+1},...,y_{n+m-1},z) $ is an $\epsilon-$chain from $x$ to $z$. Similarly, let $(x_{0}=y,x_{1},..,x_{k}=x) $ be an $\epsilon-$chain from $y$ to $x$. Then $(z, y_{n+1},...,y_{n+m}=y=x_{0},x_{1},...,y_{n}=x) $ is an $\epsilon-$chain from $z$ to $x$. So, $x$ and $z$ are $\epsilon-$chain equivalent.
\end{proof}
\section{main results}
In this section by using adding machine maps and topological conjugacy we investigate the relation between chain mixing, chain transitivity and chain recurrent properties in iterated function systems.
\begin{definition}\label{df}(Adding machine map. \cite{[BK]})
Let $\alpha=(j_{1},j_{2},...)$ be a sequence of integers where each $j_{i}\geq 2$. Let $\Delta_{\alpha}$ denote all sequences $(x_{1},x_{2},...)$ where $x_{i}\in\{0,1,...,j_{i}-1\}$ for each $i$. We put a metric $d_{\alpha}$ on $\Delta_{\alpha}$ given by $d_{\alpha}((x_{1},x_{2},...),(y_{1},y_{2},...))=\Sigma_{i=1}^{\infty}\frac{\delta(x_{i},y_{i})}{2^{i}},$ where $\delta(x_{i},y_{i})=1$ if $x_{i}\neq y_{i}$ and $\delta(x_{i},y_{i})=0$ if $x_{i}= y_{i}.$ Addition in $\Delta_{\alpha}$ is defined as follows:\\
$$(x_{1},x_{2},...)+(y_{1},y_{2},...)=(z_{1},z_{2},...),$$ where $z_{1}=(x_{1}+y_{1})~ mod ~j_{1}$ and $z_{1}=(x_{2}+y_{2}+t_{1})~ mod ~j_{2}$. Here $t_{1}=0$ if $x_{1}+y_{1}<j_{1}$ and $t_{1}=1$ if $x_{1}+y_{1}\geq j_{1}$. Continue adding and carrying in this way for the whole sequence.\\ We define, the adding machine map, $g_{\alpha}:\Delta_{\alpha}\rightarrow \Delta_{\alpha}$ by $$g_{\alpha}((x_{1},x_{2},...))=(x_{1},x_{2},...)+(1,0,0,...).$$
\end{definition}
\begin{lemma}
Let $\mathcal{F}$ be chain transitive IFS and let $\epsilon>0$. There exists $k_{\epsilon}>1$ such that for any $x\in X$, $k_{\epsilon}$ is the greatest common divisor of the length of all $\epsilon-$chains from $x$ to itself.
\end{lemma}
\begin{proof}
Similar to proof of Lemma 7 in \cite{[RW]}, it is sufficient to consider $\epsilon-$chains for $\mathcal{F}$.
\end{proof}
Let $\mathcal{F}$ be chain transitive. Define a relation on $X$ by setting $a\sim_{\epsilon}b$, if there is an $\epsilon-$chain from $a$ to $b$ of length $mk_{\epsilon}$, for some $m>0$. Since $\mathcal{F}$ is chain transitive this is clear that $\sim_{\epsilon}$ is an equivalence relation.
\begin{lemma}\label{lec}
For each $\epsilon>0$ there are $k_{\epsilon}$ equivalence classes for relation $\sim_{\epsilon}$, such that for every $\sigma\in \Lambda^{\mathbb{Z}_{+}}$, $\mathcal{F}_{\sigma}$ cycles among the classes periodically and each classes is invariant under $\mathcal{F}_{\sigma_{k_{\epsilon}}}$.
\end{lemma}
\begin{proof}
The case $k_{\epsilon}=1$ is clear.\\
Let $k_{\epsilon}>1$ and $\sigma=\{\lambda_{1},\lambda_{2},...,\lambda_{k},...\}$ be an arbitrary sequence in $\Lambda$. Consider $x\in X$, we have an $\epsilon-$chain from $x$ to $x$ of $mk_{\epsilon}$ length. If $f_{\lambda_{1}}(x)\in \overline{x}$ then we have an $\epsilon-$chain from $f_{\lambda_{1}}(x)$ to $x$ of $mk_{\epsilon}$ length, for some $m_{1}>1$. So, there exists an $\epsilon-$chain from $x$ to $x$ of $m_{1}k_{\epsilon}+1$ length, that is a contradict with $k_{\epsilon}>1$. Similarly for each $i<k_{\epsilon}$ if $\mathcal{F}_{\sigma_{i}}(x)\in \overline{x}$, then we have an $\epsilon-$chain from $x$ to $x$ of $m_{1}k_{\epsilon}+i$ length, that is a contradict. Then, $\overline{x}, \overline{f_{\lambda_{1}}(x)}, ...,\overline{\mathcal{F}_{\sigma_{i}}(x)}$ are $k_{\epsilon}$ separated equivalence classes.\\
On the other hand $\{x, f_{\lambda_{1}}(x), ...,\mathcal{F}_{\sigma_{k_{\epsilon}}}(x)\}$ is an   $\epsilon-$chain from $x$ to $x$. Then $\mathcal{F}_{\sigma_{k_{\epsilon}}}(x)\in \overline{x}$ and the proof is complete.
\end{proof}
\begin{remark}\label{rea}
Let $\epsilon_{1}\leq \epsilon_{2}$. We have the following properties:
$(a)$ Since every $\epsilon_{1}-$chain is also an $\epsilon_{2}-$chain then every $\sim_{\epsilon_{2}}$ equivalence classes is  the union of some $\sim_{\epsilon_{1}}$ equivalence classes. So, $k_{\epsilon_{2}}$ divided $k_{\epsilon_{1}}$. \\ $(b)$ Let $C_{\epsilon}=\{\overline{x_{1}},...,\overline{x_{k_{\epsilon}}}\}$ be the set of all $\sim_{\epsilon}$ equivalence classes, then $\overline{x_{1}},...,\overline{x_{k_{\epsilon}}}$ are close and open, pairwise disjoint, nonempty subset of $X$.
\end{remark}
\begin{lemma}\label{led}
Let $\{\epsilon_{i}\}$ be a sequence decreasing to $0$ and $\{k_{\epsilon_{i}}\}$ be strictly increasing of positive integer numbers. Suppose that for each positive integer $i$ there is a cover $C_{i}$ of $X$ with the following properties:\\
 $(1)$ For each $i$, $C_{i}$ contains $\{k_{\epsilon_{i}}\}$ pairwise disjoint, nonempty, clopen sets which are cyclically by $\mathcal{F}_{\sigma}.$\\
 $(2)$ $C_{i+1}$ partitions $C_{i}$, for all $i\geq 1$.\\
 $(3)$ If $V_{1}\supset V_{2}\supset V_{3}\supset ...$ is a nested sequence with $V_{i}\in C_{i}$, then $\cap_{i=1}^{\infty}V_{i}$ consists of a single point.\\
 Then for every  $\sigma=\{\lambda_{1},\lambda_{2},\lambda_{3},...\}\in \Lambda^{\mathbb{Z}_{+}}$,  $\mathcal{F}_{\sigma}$ is topologically conjugate to $g_{\alpha}$ where $\alpha=\{k_{\epsilon_{1}},\frac{k_{\epsilon_{2}}}{k_{\epsilon_{1}}},\frac{k_{\epsilon_{3}}}{k_{\epsilon_{2}}},...\}$
\end{lemma}
\begin{proof}
 Take $Q_{1}=\{Y_{1,0},...,Y_{1,k_{\epsilon_{1}}}\}$ and $Q_{i}=\{Y_{i,0},...,Y_{i,\frac{k_{\epsilon_{i}}}{k_{\epsilon_{i-1}}}}\}$, for every $i>1$ as constructed in proof of Theorem 2.3. of \cite{[bk]} which $\cap_{i\geq 1}Y_{i,t_{i}}$ consists a single point, where $t_{i}\in \{0,...,\frac{k_{\epsilon_{i}}}{k_{\epsilon_{i-1}}}-1\}$, for all $i\geq 1$.\\
 We define $h:X\rightarrow \Delta_{\alpha}$ as follows:\\ Let $x\in X $. For each positive integer $i$, there is a unique $m_{i}\leq k_{\epsilon_{i}}-1$,  such that $x\in X_{i,m_{i}}$. Now $\cap_{i\geq 1}Y_{i,m_{i}}$ consists a single point $y$. Let $h(x)=y$. By proof of Theorem 2.3. of \cite{[bk]}, $h$ is continuous and bijective. Now, we show that
$h~o\mathcal{F}_{\sigma_{n}}=g_{\alpha}^{n}o~h$ for all $n\geq 1$. \\
Let $x\in X$ be arbitrary. For each positive integer $i$ there is unique $m_{i}\leq k_{\epsilon_{i}}-1$,  such that $x\in X_{i,m_{i}}$. There is a unique $t_{i}\leq k_{\epsilon_{i}}-1$ defined by $t_{i}=(m_{i}+1) mod~k_{\epsilon_{i}}$. Then each of the points $h~o\mathcal{F}_{\sigma_{n}}$ and $g_{\alpha}^{n}o~h$ is an element of $\cap_{i\geq 1}Y_{i,t_{i}}$. Since this intersection consists of a single point, we have $h~o\mathcal{F}_{\sigma_{1}}(x)=g_{\alpha}o~h(x).$
\\ Replacing $f_{\lambda_{1}}(x)$ by $x$ and and repeat the above technique,  $$h~o\mathcal{F}_{\sigma_{2}}(x)=h(f_{\lambda_{2}}(f_{\lambda_{1}}(x)))=
g_{\alpha}o~h((f_{\lambda_{1}}(x)))=g_{\alpha}(g_{\alpha}(h(x))).$$
So, by induction on $n$ we have
$$h(\mathcal{F}_{\sigma_{n}}(x))=h(f_{\lambda_{n}}(\mathcal{F}_{\sigma_{n-1}}(x)))=
g_{\alpha}(h(\mathcal{F}_{\sigma_{n-1}}(x)))=g_{\alpha}(g_{\alpha}^{n-1}(h(x)))=g_{\alpha}^{n}(h(x)).$$
\end{proof}
The following theorem is one of the main results of this paper.
\begin{theorem}\label{main}
Let $\mathcal{F}$ be a chain transitive IFS, then either\\
$a-$ there is a $k\geq 1$ such that $\mathcal{F}$ cyclically permutes $k$ closed and open equivalence classes of $X$ and $\mathcal{F}^{k}$ restricted to each equivalence classes is chain mixing; or\\
$b-$  $\mathcal{F}$ factors to an adding machine map.
\end{theorem}
\begin{proof}
Fix $\epsilon>0$. By Lemma \ref{lec}, $\mathcal{F}$ cycles among the classes periodically and each class is invariant under $\mathcal{F}_{\sigma_{k}}$, for every $\sigma\in \Lambda^{\mathbb{Z}_{+}}$. The quantity $k_{\epsilon}$ is nondecreasing as $\epsilon\rightarrow 0$ and $\epsilon_{1}\leq\epsilon_{2}$ implies that $k_{\epsilon_{2}}$ divides $k_{\epsilon_{1}}$.\\
Now we have two cases: $\{k_{\epsilon}\}_{\epsilon>0}$ is bounded or $k_{\epsilon}$ grows without bound.\\
\textbf{Case A:} $k_{\epsilon}$ stabilizes at some $k$.\\
If $k=1$ then there is only one equivalence classes and similar arguments to proof of Theorem 6 in \cite{[RW]} shows that   $\mathcal{F}$ is chain mixing. \\
If $k=k_{\epsilon_{0}}>1$ for some $\epsilon_{0}>0$, then the equivalence relation $\sim$ is the same as $\sim_{\epsilon_{0}}$. Thus there are $k$ equivalence classes  and $\mathcal{F}$ cycles among the classes periodically and each class is invariant under $\mathcal{F}_{\sigma_{k}}$, for every $\sigma\in \Lambda^{\mathbb{Z}_{+}}$ by Lemma \ref{lec}, each equivalence class is invariant under $\mathcal{F}_{\sigma_{k}}$ and similar to the case $k=1$, $\mathcal{F}_{\sigma_{k}}$ is chain mixing on each equivalence class.\\ \textbf{Case B:} Let $k_{\epsilon}$ grows without bound and $\sigma\in \Lambda^{\mathbb{Z}_{+}}$ be arbitrary. So, we have a sequence $\{\epsilon_{i}\}$ of positive numbers decreasing to $0$ and the related sequences $\{k_{\epsilon_{i}}\}$ that is strictly increasing. Then for every $\{\epsilon_{i}\}$ we have $\{k_{\epsilon_{i}}\}$ equivalence classes, and by Remark \ref{rea} and Lemma \ref{led}  $\mathcal{F}_{\sigma}$ is topologically conjugate to $g_{\alpha}$, where $\alpha=\{k_{\epsilon_{1}},\frac{k_{\epsilon_{2}}}
{k_{\epsilon_{1}}},\frac{k_{\epsilon_{3}}}{k_{\epsilon_{2}}},...\}$.
\end{proof}
Now we obtain the main result of this paper.
\begin{theorem}\label{ca}
Let $X$ be connected and $\mathcal{F}=\{X; f_{\lambda}|\lambda\in\Lambda\}$ be an IFS, consists continuous functions. Then the following are equivalent:\\
$(1)~\mathcal{F}$ is chain mixing.\\
$(2)~\mathcal{F}$ is totally chain transitive.\\
$(3)~\mathcal{F}$ is  chain transitive.\\
$(4)~\mathcal{F}$ is chain recurrent.\\
\end{theorem}
\begin{proof}
By definitions this is clear that $(1)\Rightarrow (2)\Rightarrow (3)\Rightarrow (4).$ Since $X$ is connected, by Theorem \ref{main} we have $(3)\Rightarrow (1)$ and by  Lemma \ref{lea} $(4)\Rightarrow (3)$, so the proof is complete.
\end{proof}
Lemma \ref{le} and Theorem \ref{ca} implies the following result.
\begin{corollary}
Let $X$ be connected and $\mathcal{F}=\{X; f_{\lambda}|\lambda\in\Lambda\}$ be an IFS, consists continuous functions which has the shadowing property. Then the following are equivalent:\\
$(1)~\mathcal{F}$ is topologically mixing.\\
$(2)~\mathcal{F}$ is totally chain transitive.\\
$(3)~\mathcal{F}$ is  chain transitive.\\
$(4)~\mathcal{F}$ is chain recurrent.\\
\end{corollary}
Lemma \ref{lex} and Theorem \ref{ca} implies the following result.
\begin{corollary}
Let $X$ be connected and $\mathcal{F}=\{X; f_{\lambda}|\lambda\in\Lambda\}$ be an IFS, consists continuous functions and that $ \mathcal{F}^{n}$ is chain transitive, for all positive integers $n$. Then we have the following properties:\\
$(1)~\mathcal{F}\times  \mathcal{F}$ is chain  mixing.\\
$(2)~\mathcal{F}\times  \mathcal{F}$ is totally chain transitive.\\
$(3)~\mathcal{F}\times  \mathcal{F}$ is chain recurrent.\\
\end{corollary}

\end{document}